\documentclass{amsart}
\usepackage{ucs}
\usepackage{verbatim}
\usepackage{paralist}
\usepackage{leftidx}
\usepackage{dsfont}
\usepackage{amsthm}
\usepackage{amsmath}
\usepackage{amssymb}
\usepackage{amscd}
\usepackage{graphicx}
\usepackage{setspace}
\usepackage[all]{xy}
\usepackage{mathrsfs} 
\usepackage{hyperref}
\usepackage{mathtools}
\title[Intersection numbers as mixed volumes]{Intersection numbers as mixed volumes of Newton--Okounkov bodies}
\author{Robert Wilms}
\address{Universit\'e de Caen Normandie, CNRS, LMNO UMR 6139, F-14000 Caen, France}
\email{\href{mailto:robert.wilms87@gmail.com}{robert.wilms87@gmail.com}}
\date{\today}
\begin{document}
	\numberwithin{equation}{section}
	\newtheorem{Def}{Definition}
	\numberwithin{Def}{section}
	\newtheorem{Rem}[Def]{Remark}
	\newtheorem{Lem}[Def]{Lemma}
	\newtheorem{Que}[Def]{Question}
	\newtheorem{Cor}[Def]{Corollary}
	\newtheorem{Exam}[Def]{Example}
	\newtheorem{Thm}[Def]{Theorem}
	\newtheorem*{clm}{Claim}
	\newtheorem{Pro}[Def]{Proposition}
	\newcommand\gf[2]{\genfrac{}{}{0pt}{}{#1}{#2}}
	
\begin{abstract}
	In this paper we express any intersection number $(L_1\cdot\ldots\cdot L_d)$ of ample line bundles on an irreducible projective variety as the mixed volume $V(\Delta_{Y_\bullet}(L_1),\dots,\Delta_{Y_\bullet}(L_d))$ of their Newton--Okounkov bodies. The admissible flag $Y_\bullet$ of subvarieties is constructed from sections of the line bundles using Bertini's theorem, allowing some flexibility to vary the line bundles after the flag is fixed. The proof relies on the slice formula for Newton--Okounkov bodies and on mixed-volume calculations in convex geometry.
\end{abstract}
\maketitle
\section{Introduction}
Newton--Okounkov bodies have been introduced independently by Lazarsfeld--Mustaţă \cite{LM09} and Kaveh--Khovanskii \cite{KK12} based on ideas by Okounkov \cite{Oko96,Oko03} to study line bundles in algebraic geometry by convex geometric methods.
Throughout this paper, let $X$ be an irreducible projective variety of dimension $d$ defined over an algebraically closed field of any characteristic. The Newton--Okounkov body of a big line bundle $L$ on $X$ is a convex body $\Delta_{Y_{\bullet}}(L)\subseteq\mathbb{R}^{d}$ that encodes a lot of information about $L$, for example its volume. Its construction depends on the choice of an admissible flag $Y_\bullet$ on $X$, that is, a flag
\begin{align*}
	Y_\bullet\colon \qquad X=Y_0\supsetneq Y_1\supsetneq Y_2\supsetneq \dots\supsetneq Y_d=\{\text{pt.}\}
\end{align*}
of irreducible subvarieties such that the point $Y_d$ is regular on $Y_k$ for all $0\le k\le d-1$. By Jow's theorem \cite[Theorem A]{Jow10} and by \cite[Proposition 4.1(i)]{LM09}, two big line bundles $L_1$ and $L_2$ are numerically equivalent if and only if $\Delta_{Y_\bullet}(L_1)=\Delta_{Y_\bullet}(L_2)$ for all admissible flags $Y_\bullet$. This raises the question \cite{Xia21}: How can intersection numbers of line bundles be described in terms of their Newton--Okounkov bodies? In this paper, we answer this question in the case of ample line bundles.
\begin{Thm}\label{thm_main}
	Let $X$ be any irreducible projective variety of dimension $d$ and $L_1,\dots, L_d$ be ample $\mathbb{Q}$-line bundles on $X$. There exists an admissible flag $Y_\bullet$ on $X$ such that
	$$(L_1\cdot\ldots\cdot L_d)=d!V(\Delta_{Y_\bullet}(L_1),\dots,\Delta_{Y_\bullet}(L_d)),$$
	where $(L_1\cdot \ldots \cdot L_d)$ is the intersection number of the line bundles $L_1,\dots, L_d$ and $V(\Delta_{Y_\bullet}(L_1),\dots,\Delta_{Y_\bullet}(L_d))$ is the mixed volume of $\Delta_{Y_\bullet}(L_1),\dots,\Delta_{Y_\bullet}(L_d)$.
\end{Thm}
The theorem is a special case of the following more general and more explicit result.
\begin{Thm}\label{thm}
	Let $X$ be any irreducible projective variety of dimension $d$ and $L_1,\dots, L_{k}, L$ be ample $\mathbb{Q}$-line bundles on $X$ with $0\le k\le d-1$. There exists an admissible flag $Y_\bullet$ on $X$ with the following property: For any $\mathbb{Q}$-line bundle $M$ on $X$ and any
	$$L_{k+1},\dots,L_d\in C_L(M)=\{\lambda L+\mu M~|~\lambda\in \mathbb{R}, \mu\in\mathbb{R}_{\ge 0}\}\cap \mathrm{Amp}(X),$$
	where $\mathrm{Amp}(X)\subseteq N^1(X)_{\mathbb{R}}$ denotes the ample cone, we have
	\begin{align}\label{equ_mixed-intersection}
	(L_1\cdot \ldots \cdot L_d)=d!V(\Delta_{Y_\bullet}(L_1),\dots,\Delta_{Y_\bullet}(L_d)).
	\end{align}
	Moreover, even if $L$ is not ample, every admissible flag with the properties
	\begin{enumerate}[(i)]
		\item $Y_{j}\subseteq Y_{j-1}$ is a Cartier divisor for every $1\le j\le d-1$.
		\item For every $1\le j\le k$ there exists an $r_{j}\in \mathbb{Q}_{>0}$ such that $r_{j}\mathcal{O}_{Y_{j-1}}(Y_{j})\equiv L_{j}|_{Y_{j-1}}$, where $\equiv$ denotes numerical equivalence of $\mathbb{Q}$-line bundles.
		\item For every $k+1\le j\le d-1$ there is an $r_j\in\mathbb{Q}$ such that $r_j\mathcal{O}_{Y_{j-1}}(Y_j)\equiv L|_{Y_{j-1}}$.
	\end{enumerate} 
	satisfies Equation (\ref{equ_mixed-intersection}) for all $L_{k+1},\dots,L_d\in C_L(M)$ for any $\mathbb{Q}$-line bundle $M$.
\end{Thm}
To prove the theorem, we first construct a flag satisfying properties (i)-(iii) by Bertini's theorem. 
Then we show Equation (\ref{equ_mixed-intersection}) by induction on $k$, making use of the slice formula for Newton--Okounkov bodies. The case $k=0$ was proved in \cite{Wil25}. We will also need a lemma about the compatibility of mixed volumes of certain convex bodies with restriction to the hyperplane $\{x_1=0\}$, which we prove in Section \ref{sec_mv}.

A main motivation of Theorem \ref{thm_main} is to translate general inequalities between mixed volumes of convex bodies into inequalities between intersection numbers of nef line bundles. As a first application, we will derive the following corollary.
\begin{Cor}\label{cor}
	Let $X$ be any irreducible projective variety of dimension $d$ and $L, L_1,\dots L_{d}$ be nef $\mathbb{R}$-line bundles on $X$.
	Then $$(L_1\cdot\ldots\cdot L_d)\cdot (L^d)\le d\cdot(L^{d-1}\cdot L_d)\cdot(L_1\cdot\ldots\cdot L_{d-1}\cdot L).$$
\end{Cor}
We remark that a more general inequality has been proven by Jiang--Li \cite[Theorem 3.5]{JL23} by different methods using multipoint Okounkov bodies.
\section{Newton--Okounkov Bodies}
In this section we review the definition and main properties of Newton--Okounkov bodies. We refer to \cite{LM09} for more details.
Let $Y_\bullet$ be an admissible flag on $X$ and $L$ a line bundle on $X$. There is a valuation-like function
$$\nu_{Y_\bullet}\colon H^0(X,L)\setminus \{0\}\to \mathbb{Z}^d,\qquad s\mapsto \nu_{Y_\bullet}(s)=(\nu_1(s),\dots,\nu_d(s)),$$
which is defined inductively by $\nu_j(s)=\mathrm{ord}_{Y_j}(s_{j-1})$, where $s_0=s$ and $s_j$ is the section of the line bundle $L|_{Y_j}\otimes \bigotimes_{i=1}^j\mathcal{O}_{Y_{i-1}}(-\nu_i(s)Y_i)|_{Y_j}$ on $Y_j$ induced by $s_{j-1}$ for $j\ge 1$. 
The Newton--Okounkov body $\Delta_{Y_\bullet}(L)$ of $L$ is defined by
$$\Delta_{Y_\bullet}(L)=\mathrm{cch}\left(\bigcup_{m\ge 1}\tfrac{1}{m}\nu_{Y_\bullet}\left(H^0(X,L^{\otimes m})\setminus\{0\}\right)\right)\subseteq \left(\mathbb{R}_{\ge 0}\right)^d,$$
where $\mathrm{cch}$ denotes \emph{closed convex hull}. 

Let us recall some facts about Newton--Okounkov bodies. For this purpose, we assume that $L$ is big. By \cite[Theorem 2.3]{LM09} we have the following formula for the volume
\begin{align}\label{equ_volume}
	\mathrm{vol}(\Delta_{Y_\bullet}(L))=\tfrac{1}{d!}\mathrm{vol}(L)=\lim_{m\to\infty}\frac{\dim H^0(X,L^{\otimes m})}{m^d}.
\end{align}
Note that $\mathrm{vol}(L)=(L^d)$ if $L$ is ample.
For any integer $p>0$, we have
\begin{align}\label{equ_scalar}
	\Delta_{Y_\bullet}(L^{\otimes p})=p\cdot \Delta_{Y_\bullet}(L)
\end{align}
as shown in \cite[Proposition 4.1]{LM09}. Hence, the definition of Newton--Okounkov bodies extends canonically to big $\mathbb{Q}$-line bundles. Since this extension is continuous on the cone of big $\mathbb{Q}$-line bundles by \cite[Theorem B]{LM09}, Newton--Okounkov bodies are therefore canonically defined for big $\mathbb{R}$-line bundles. By homogeneity and continuity, Equation (\ref{equ_scalar}) holds true for all $p\in\mathbb{R}_{>0}$ and all big $\mathbb{R}$-line bundles $L$.

By sending two global sections $s_1\in H^0(X,L_1^{\otimes m})$ and $s_2\in H^0(X,L_2^{\otimes m})$ of big line bundles $L_1,L_2$ to $s_1\otimes s_2\in H^0(X,(L_1\otimes L_2)^{\otimes m})$ we get an inclusion
\begin{align}\label{equ_subadditivity}
	\Delta_{Y_\bullet}(L_1)+\Delta_{Y_\bullet}(L_2)\subseteq \Delta_{Y_\bullet}(L_1\otimes L_2),	
\end{align}
where the left-hand side denotes the Minkowski sum of $\Delta_{Y_\bullet}(L_1)$ and $\Delta_{Y_\bullet}(L_2)$. Note again, that by homogeneity and continuity this inclusion holds true for all big $\mathbb{R}$-line bundles $L_1$ and $L_2$.

If $d\ge 1$, we also consider restrictions to $Y_1$. First, let us denote the restricted flag 
$$Y_{1,\bullet}\colon \qquad Y_1\supsetneq Y_2\supsetneq \dots\supsetneq Y_d$$
on the projective variety $Y_1$. We define the restricted Newton--Okounkov body $\Delta_{Y_\bullet|Y_1}(L)$ in $\mathbb{R}^{d-1}$ by
$$\Delta_{Y_\bullet|Y_1}(L)=\mathrm{cch}\left(\bigcup_{m\ge 1}\tfrac{1}{m}\nu_{Y_{1,\bullet}}\left(\mathrm{Im}\left(H^0(X,L^{\otimes m})\to H^0(Y_1,L^{\otimes m}|_{Y_1})\right)\setminus\{0\}\right)\right).$$
If $L$ is ample, then the restriction map is surjective for $m\gg 0$. Hence, we get
\begin{align}\label{equ_restriction-ample}
	\Delta_{Y_\bullet|Y_1}(L)=\Delta_{Y_{1,\bullet}}(L|_{Y_1}).
\end{align}
By homogeneity, this formula is still true for ample $\mathbb{Q}$-line bundles and by continuity it is also true for ample $\mathbb{R}$-line bundles.

Let us recall the slice formula for Newton--Okounkov bodies from \cite[Theorem 4.26]{LM09}. 
We assume that $L$ is an ample $\mathbb{R}$-line bundle and $Y_1\subseteq X$ is a Cartier divisor. We write
$$\mu(L;Y_1)=\sup\{s>0~|~L-s\mathcal{O}_X(Y_1) \text{ is big.}\}.$$
The slice formula asserts that, for $\tau\ge 0$ with $\tau\neq \mu(L;Y_1)$,
\begin{align}\label{equ_slice-formula}
	\Delta_{Y_{\bullet}}(L)_{\nu_1=\tau}=\Delta_{Y_\bullet|Y_1}(L-\tau\mathcal{O}_X(Y_1))
\end{align}
considered as convex bodies in $\mathbb{R}^{d-1}\cong \{\tau\}\times \mathbb{R}^{d-1}$.
Note that for $\tau>\mu(L;Y_1)$ both sides are empty sets.

\section{Mixed volumes}\label{sec_mv}
In this section we prove a lemma about mixed volumes that will be needed in the proof of Theorem \ref{thm}.
For details on mixed volumes of convex bodies, we refer to Schneider's book \cite[Chapter 5]{Sch14}.
For any convex bodies $K_1,\dots,K_n$ in $\mathbb{R}^d$ the volume of the Minkowski sum $t_1K_1+\dots +t_nK_n$ is a homogeneous polynomial of degree $d$ in $t_1,\dots, t_n$ if $t_i\ge 0$ for all $i$. The mixed volumes $V(K_{i_1},\dots,K_{i_d})$ for $(i_1,\dots, i_d)\in \{1,\dots,n\}^d$ are defined by the coefficients of this polynomial, that is,
\begin{align}\label{equ_def_mv}
\mathrm{vol}(t_1K_1+\dots+t_nK_n)=\sum_{i_1,\dots, i_d=1}^n t_{i_1} \cdots t_{i_d}V(K_{i_1},\dots, K_{i_d})
\end{align}
for all $t_1\ge 0,\dots,t_n\ge 0$, where $V(K_{i_1},\dots,K_{i_d})$ is symmetric in $K_{i_1},\dots,K_{i_d}$. It turns out that $V(\cdot,\ldots,\cdot)$ is a non-negative and multilinear function on the cone of convex bodies in $\mathbb{R}^d$ and that $V(K,\dots,K)=\mathrm{vol}(K)$.
For every $\tau\in\mathbb{R}$, let $H_{\tau}$ denote the hyperplane
$$H_\tau=\{(x_1,\dots, x_d)\in\mathbb{R}^d~|~x_1=\tau\}.$$
For any convex body $K\subseteq \mathbb{R}^d$ we write $K_\tau$ for $K\cap H_\tau$ considered as a convex body in $\mathbb{R}^{d-1}\cong H_\tau$, that is, $K\cap H_\tau=\{\tau\}\times K_\tau$. For simplicity, we also call the empty set a convex body.
\begin{Lem}\label{lem_mixed-volumes}
	Let $K_1,\dots, K_d$ be convex bodies in $(\mathbb{R}_{\ge 0})^d$ and $r\in\mathbb{R}_{>0}$. If
	\begin{enumerate}[(a)]
		\item $K_{1,\tau}=\left(1-\tfrac{\tau}{r}\right)K_{1,0}$ for all $\tau\in [0,r)$ and
		\item $\tfrac{\tau}{r}K_{1,0}+K_{j,\tau}\subseteq K_{j,0}$ for all $2\le j\le d$ and all $\tau\ge 0$,
	\end{enumerate}
	then it holds $\frac{d}{r} V(K_1,\dots, K_d)=V(K_{2,0},\dots, K_{d,0})$.
\end{Lem}
\begin{proof}
The idea is to compute the volume of $t_1K_1+\dots+t_dK_d$ slice-wise and to compare the coefficients of $t_1\cdots t_d$. In what follows, we assume that $t_i>0$ for all $i\in\{1,\dots,d\}$. By the definition of the mixed volume in Equation (\ref{equ_def_mv}) we have
\begin{align}\label{equ_computation-integral}
	&\sum_{i_1,\dots,i_d=1}^d t_{i_1}\cdots t_{i_d} V(K_{i_1},\dots, K_{i_d})=\int_{0}^\infty \mathrm{vol}_{\mathbb{R}^{d-1}}((t_1 K_1+\dots+t_dK_d)\cap H_\tau)d\tau.
\end{align}	
We can decompose the slice of the Minkowski sum by
$$(t_1 K_1+\dots+t_dK_d)\cap H_\tau=\bigcup_{\tau'\in [0,\tau]} \left((t_1 K_1)\cap H_{\tau'}+(t_2K_2+\dots+t_dK_d)\cap H_{\tau-\tau'}\right).$$
By assumption (a) we have
$$(t_1K_1)\cap H_{\tau'}=\{\tau'\}\times (t_1K_1)_{\tau'}=\{\tau'\}\times t_1 K_{1,\frac{\tau'}{t_1}}=\{\tau'\}\times \left(t_1-\tfrac{\tau'}{r}\right)K_{1,0}$$
for all $\tau'\in [0,rt_1]$, where the case $\tau'=rt_1$ follows by continuity.
Next we show a variant of assumption (b) for the Minkowski sum $t_2K_2+\dots+t_dK_d$:
\begin{align*}
	&\tfrac{\tau}{r}K_{1,0}+(t_2K_2+\dots+t_dK_d)_\tau\\
	&=\tfrac{\tau}{r}K_{1,0}+\bigcup_{\tau_2+\dots+\tau_d=\tau}\left(t_2K_{2,\frac{\tau_2}{t_2}}+\dots+t_dK_{d,\frac{\tau_d}{t_d}}\right)\nonumber\\
	&=\bigcup_{\tau_2+\dots+\tau_d=\tau}\left(	\tfrac{\tau}{r}K_{1,0}+t_2K_{2,\frac{\tau_2}{t_2}}+\dots+t_dK_{d,\frac{\tau_d}{t_d}}\right)\nonumber\\
	&=\bigcup_{\tau_2+\dots+\tau_d=\tau}\left(t_2\left(\tfrac{\tau_2}{t_2r}K_{1,0}+K_{2,\frac{\tau_2}{t_2}}\right)+\dots+t_d\left(\tfrac{\tau_d}{t_dr}K_{1,0}+K_{d,\frac{\tau_d}{t_d}}\right)\right)\nonumber\\
	&\subseteq \bigcup_{\tau_2+\dots+\tau_d=\tau}(t_2K_{2,0}+\dots+t_dK_{d,0})=t_2K_{2,0}+\dots+t_dK_{d,0}, \nonumber
\end{align*}
where all $\tau_j$ are non-negative. Moreover, one checks that the subset $K\subseteq (\mathbb{R}_{\ge 0})^d$ with the slices $K_\tau=\frac{\tau}{r}K_{1,0}+(t_2K_2+\dots+t_dK_d)_\tau$ is a convex body, where the slice at $\tau=0$ contains every other slice. Thus,
\begin{align}\label{equ_inclusion}
\tfrac{\tau_1}{r}K_{1,0}+(t_2K_2+\dots+t_dK_d)_{\tau_1}\subseteq \tfrac{\tau_2}{r}K_{1,0}+(t_2K_2+\dots+t_dK_d)_{\tau_2}
\end{align}
if $\tau_1\ge \tau_2\ge 0$.

We deduce for $\tau<rt_1$ that
\begin{align*}
	&(t_1 K_1)\cap H_{\tau'}+(t_2K_2+\dots+t_dK_d)\cap H_{\tau-\tau'}\\
	&=\{\tau'\}\times \left(t_1-\tfrac{\tau'}{r}\right)K_{1,0}+\{\tau-\tau'\}\times (t_2K_2+\dots+t_dK_d)_{\tau-\tau'}\\
	&=\{\tau\}\times \left( \left(t_1-\tfrac{\tau'}{r}\right)K_{1,0}+(t_2K_2+\dots+t_dK_d)_{\tau-\tau'}\right)\\
	&\subseteq\{\tau\}\times \left(\left(t_1-\tfrac{\tau}{r}\right)K_{1,0}+t_2K_{2,0}+\dots+t_dK_{d,0}\right)
\end{align*}
for all $\tau'\in [0,\tau]$ with equality if $\tau'=\tau$. For $\tau\ge rt_1$ we similarly get by the inclusion in (\ref{equ_inclusion}) and by the cancellation law for Minkowski sums that
\begin{align*}
	&(t_1 K_1)\cap H_{\tau'}+(t_2K_2+\dots+t_dK_d)\cap H_{\tau-\tau'}\\
	&=\{\tau\}\times \left( \left(t_1-\tfrac{\tau'}{r}\right)K_{1,0}+(t_2K_2+\dots+t_dK_d)_{\tau-\tau'}\right)\\
	&\subseteq \{\tau\} \times\left((t_2K_2+\dots+t_dK_d)_{\tau-t_1r}\right)
\end{align*}
for all $\tau'\in[0,rt_1]$ with equality if $\tau'=rt_1$. Note that for $\tau'>rt_1$ the Minkowski sum in the first line is the empty set.

We can now apply the above observations to the computation of the integral in Equation (\ref{equ_computation-integral}). This yields
\begin{align*}
	&\sum_{i_1,\dots,i_d=1}^dt_{i_1}\cdots t_{i_d} V(K_{i_1},\dots, K_{i_d})\\
	&=\int_0^{rt_1}\mathrm{vol}\left(\left(t_1-\tfrac{\tau}{r}\right)K_{1,0}+t_2K_{2,0}+\dots+t_dK_{d,0}\right)d\tau\\
	&\quad+\int_{rt_1}^\infty \mathrm{vol}\left((t_2K_2+\dots+t_dK_d)_{\tau-t_1r}\right)d\tau\\
	&=\int_0^{rt_1}\sum_{i_1,\dots,i_{d-1}=1}^d\widetilde{t}_{i_1}\cdots \widetilde{t}_{i_{d-1}}V(K_{i_1,0},\dots,K_{i_{d-1},0})d\tau\\
	&\quad +\int_0^\infty  \mathrm{vol}\left((t_2K_2+\dots+t_dK_d)_{\tau}\right)d\tau,
\end{align*}
where $\widetilde{t}_1=t_1-\frac{\tau}{r}$ and $\widetilde{t}_i=t_i$ for $2\le i\le d$. We only get a $(t_1\cdots t_d)$-term in the last expression if we consider the first integral and the summands with $\{i_1,\dots, i_{d-1}\}=\{2,\dots,d\}$. Thus,
$$d! V(K_1,\dots,K_d)=r(d-1)!V(K_{2,0},\dots, K_{d,0}).$$
Dividing by $r(d-1)!$ on both sides gives the statement of the lemma.
\end{proof}

\section{Construction of the flag}\label{sec_existence}
In this section, we show the existence of an admissible flag $Y_\bullet$ satisfying the properties (i)-(iii) in Theorem \ref{thm}. For simplicity, we set
$$L'_j=\begin{cases} L_j&\text{if } 1\le j\le k,\\ L&\text{if } k+1\le j\le d-1.\end{cases}$$
We proceed as in \cite[Section 4]{Wil25}. We show by induction that for every $0\le j\le d-1$ there is a flag of irreducible subvarieties $X=Y_0\supsetneq Y_1\supsetneq\dots\supsetneq Y_j$ such that for all $0\le i\le j-1$ the following hold
\begin{itemize}
	\item $Y_{i+1}$ is an irreducible Cartier divisor on $Y_{i}$,
	\item $Y_{i+1}\cap Y_{i,\mathrm{reg}}\subseteq Y_{i+1,\mathrm{reg}}$,
	\item there is an $r_{i+1}\in\mathbb{Q}$ with $r_{i+1}\mathcal{O}_{Y_{i}}(Y_{i+1})\cong L'_{i+1}|_{Y_{i}}$ as $\mathbb{Q}$-line bundles, and
	\item $Y_j\cap X_{\mathrm{reg}}\neq \emptyset$.	
\end{itemize}
This is trivial for $j=0$, and we may assume that we have constructed such a flag for some $0\le j\le d-2$. We must show that we can extend it by choosing $Y_{j+1}$ so that the properties hold for $i=j$ and $Y_{j+1}\cap X_{\mathrm{reg}}\neq \emptyset$.

Since $L'_{j+1}$ is ample, $L'_{j+1}|_{Y_j}$ is ample, too. Thus, we can choose a positive integer $\widetilde{r}_{j+1}\in \mathbb{Z}_{\ge 1}$ such that $\widetilde{r}_{j+1}L'_{j+1}|_{Y_j}$ is very ample, and hence it induces an embedding $Y_{j}\to \mathbb{P}^N$. Since $\dim Y_{j}=d-j\ge 2$, we can use the following two versions of Bertini's theorem: By \cite[Theorem 1.1]{FL81} any general
hyperplane $H\subseteq \mathbb{P}^{N}$ intersects $Y_j$ in an irreducible Cartier divisor on $Y_j$. By \cite[Corollary 2]{CGM86} we have
$(H\cap Y_j)_{\mathrm{reg}}\supseteq Y_{j,\mathrm{reg}}\cap H$ for any general hyperplane $H\subseteq\mathbb{P}^{N}$. Since every hyperplane $H\subseteq \mathbb{P}^{N}$ intersects $Y_j$, any general hyperplane $H\subseteq \mathbb{P}^{N}$ intersects the dense open subset $Y_j\cap X_{\mathrm{reg}}$ of $Y_j$.
Hence, there is a hyperplane $H_1\subseteq \mathbb{P}^{N}$ such
that $Y_{j+1}=H_1\cap Y_j$ is an irreducible Cartier divisor on $Y_j$, every point of $Y_{j+1}$ that is regular on $Y_j$ is also regular on $Y_{j+1}$, and $Y_{j+1}\cap X_{\mathrm{reg}}\neq \emptyset$. By construction we have $r_{j+1}\mathcal{O}_{Y_j}(Y_{j+1})\cong L'_{j+1}|_{Y_j}$ for $r_{j+1}=\widetilde{r}_{j+1}^{-1}\in \mathbb{Q}$. Thus, we have proved the induction hypothesis. 	

To complete the flag $Y_0\supsetneq Y_1\supsetneq\dots\supsetneq Y_{d-1}$, we set $Y_d=\{p\}$ for any closed point $p\in Y_{d-1}\cap X_{\mathrm{reg}}$. By construction, we have $p\in Y_{j,\mathrm{reg}}$ for all $0\le j\le d-1$. Hence, $Y_{\bullet}$ is an admissible flag satisfying properties (i)-(iii) in Theorem \ref{thm} with $r_j=\widetilde{r}_{j}^{-1}$.

\section{The Newton--Okounkov bodies in this paper}\label{sec_situation}
The purpose of this section is to show that the Newton--Okounkov bodies in Theorem \ref{thm} satisfy conditions (a) and (b) of Lemma \ref{lem_mixed-volumes}.
\begin{Lem}\label{lem_NO-bodies}
	Let $L_1,\dots,L_d$ be ample $\mathbb{R}$-line bundles on $X$ and $Y_\bullet$ an admissible flag on $X$ satisfying the corresponding properties (i)-(iii) in Theorem \ref{thm} with $k\ge 1$. Then it holds
	\begin{enumerate}[(a)]
		\item $\Delta_{Y_\bullet}(L_1)_{\nu_1=\tau}=\left(1-\frac{\tau}{r_1}\right)\Delta_{Y_\bullet}(L_1)_{\nu_1=0}$ for all $\tau\in [0,r_1)$ and
		\item $\frac{\tau}{r_1}\Delta_{Y_\bullet}(L_1)_{\nu_1=0}+\Delta_{Y_\bullet}(L_j)_{\nu_1=\tau}\subseteq \Delta_{Y_\bullet}(L_j)_{\nu_1=0}$ for all $2\le j\le d$ and all $\tau\ge 0$.
	\end{enumerate}
\end{Lem}
\begin{proof}
	\begin{enumerate}[(a)]
		\item By the numerical equivalence $L_1\equiv r_1\mathcal{O}_X(Y_1)$ we have $\mu(L_1;Y_1)=r_1$. Thus, for $\tau\in [0, r_1)$ the slice formula (\ref{equ_slice-formula}) implies
		\begin{align}\label{equ_computation-a}
		\Delta_{Y_\bullet}(L_1)_{\nu_1=\tau}=\Delta_{Y_\bullet|Y_1}(L_1-\tau\mathcal{O}_X(Y_1))=\Delta_{Y_{1,\bullet}}\left(\left(1-\tfrac{\tau}{r_1}\right)L_1|_{Y_1}\right),
		\end{align}
		where the second equality follows by Equation (\ref{equ_restriction-ample}) and the numerical equivalence $L_1-\tau\mathcal{O}_X(Y_1)\equiv \left(1-\frac{\tau}{r_1}\right) L_1$. Using Equation (\ref{equ_scalar}) we deduce
		\begin{align*}
		\Delta_{Y_{1,\bullet}}\left(\left(1-\tfrac{\tau}{r_1}\right)L_1|_{Y_1}\right)&=\left(1-\tfrac{\tau}{r_1}\right)\Delta_{Y_{1,\bullet}}\left(L_1|_{Y_1}\right)=\left(1-\tfrac{\tau}{r_1}\right)\Delta_{Y_{\bullet}}\left(L_1\right)_{\nu_1=0},
		\end{align*}
		where the second equality follows from Equation (\ref{equ_computation-a}) applied to $\tau=0$. This proves the first part.
		\item As $L_1\equiv r_1\mathcal{O}_X(Y_1)$, we can compute by the slice formula (\ref{equ_slice-formula}), Equation (\ref{equ_scalar}), and Equations (\ref{equ_subadditivity}) and (\ref{equ_restriction-ample}) for $\tau\neq \mu(L_j;Y_1)$ that
		\begin{align*}
			\tfrac{\tau}{r_1}\Delta_{Y_\bullet}(L_1)_{\nu_1=0}+\Delta_{Y_\bullet}(L_j)_{\nu_1=\tau}&=\Delta_{Y_{1,\bullet}}(\tau\mathcal{O}_X(Y_1)|_{Y_1})+\Delta_{Y_\bullet|Y_1}(L_j-\tau\mathcal{O}_X(Y_1))\\
			&\subseteq\Delta_{Y_{1,\bullet}}(\tau\mathcal{O}_X(Y_1)|_{Y_1})+\Delta_{Y_{1,\bullet}}((L_j-\tau\mathcal{O}_X(Y_1))|_{Y_1})\\
			&\subseteq\Delta_{Y_{1,\bullet}}(L_j|_{Y_1})=\Delta_{Y_{\bullet}}(L_j)_{\nu_1=0}. 
		\end{align*}
	For the first inclusion, note that we always have $\Delta_{Y_\bullet|Y_1}(L)\subseteq \Delta_{Y_{1,\bullet}}(L|_{Y_1})$ for any line bundle $L$ on $X$. By continuity of Newton--Okounkov bodies, the inclusion also holds for $\tau=\mu(L_j;Y_1)$. This completes the proof of the lemma.\qedhere
	\end{enumerate}
\end{proof}
\section{Proof of Theorem \ref{thm}}
We prove Theorem \ref{thm} by induction on $k$. By Section \ref{sec_existence} we only have to prove the second statement. If $k=0$, we proved in \cite[Theorem 1.2]{Wil25} that the inclusion (\ref{equ_subadditivity}) is indeed an equality $\Delta_{Y_\bullet}(N_1)+\Delta_{Y_\bullet}(N_2)=\Delta_{Y_\bullet}(N_1+N_2)$ if $N_1,N_2\in C_L(M)$. Thus, the theorem follows from the multilinearity of the intersection number and the mixed volume. See also \cite[Section 4]{Wil25}. Hence, we may assume $k\ge 1$. In particular, $d\ge 2$. Applying the induction hypothesis to $Y_1$ we get
\begin{align*}
(L_1\cdot\ldots\cdot L_d)&=r_1(L_2|_{Y_1}\cdot\ldots\cdot L_d|_{Y_1})\\
&=r_1(d-1)!V(\Delta_{Y_{1,\bullet}}(L_2|_{Y_1}),\dots,\Delta_{Y_{1,\bullet}}(L_d|_{Y_1})).
\end{align*}
Since $L_2,\dots,L_d$ are ample, we get by Equation (\ref{equ_restriction-ample}) that
$$V(\Delta_{Y_{1,\bullet}}(L_2|_{Y_1}),\dots,\Delta_{Y_{1,\bullet}}(L_d|_{Y_1}))=V(\Delta_{Y_\bullet|Y_1}(L_2),\dots,\Delta_{Y_\bullet|Y_1}(L_d)).$$
Using the slice formula (\ref{equ_slice-formula}) we get
$$V(\Delta_{Y_\bullet|Y_1}(L_2),\dots,\Delta_{Y_\bullet|Y_1}(L_d))=V(\Delta_{Y_\bullet}(L_2)_{\nu_1=0},\dots,\Delta_{Y_\bullet}(L_d)_{\nu_1=0}),$$
where we consider the convex bodies in the mixed volume on the right-hand side as convex bodies in $\mathbb{R}^{d-1}\cong \{0\}\times \mathbb{R}^{d-1}$.
By Lemmas \ref{lem_mixed-volumes} and \ref{lem_NO-bodies} we get
$$V(\Delta_{Y_\bullet}(L_2)_{\nu_1=0},\dots,\Delta_{Y_\bullet}(L_d)_{\nu_1=0})=dr_1^{-1}V(\Delta_{Y_\bullet}(L_1),\dots,\Delta_{Y_\bullet}(L_d)).$$
Thus, we conclude
$$(L_1\cdot\ldots\cdot L_d)=d!V(\Delta_{Y_\bullet}(L_1),\dots,\Delta_{Y_\bullet}(L_d))$$
as stated in the theorem.

\section{Proof of Corollary \ref{cor}}
We prove Corollary \ref{cor} in this section. By continuity, we may assume that $L,L_1,\dots,L_d$ are ample $\mathbb{Q}$-line bundles. It has been shown by Saroglou--Soprunov--Zvavitch \cite[Equation 6.1]{SSZ19} that
$$V(K_1,\dots,K_d)\mathrm{vol}(K)\le d V(K,\dots,K,K_d)V(K_1,\dots, K_{d-1},K)$$
for all convex bodies $K,K_1,\dots, K_d$ in $\mathbb{R}^d$. If $Y_\bullet$ is an admissible flag on $X$ satisfying (i)-(iii) of Theorem \ref{thm}, we get by Theorem \ref{thm} for $k=d-1$ and by Equation (\ref{equ_volume})
\begin{align*}(L_1\cdot\ldots\cdot L_d)\cdot (L^d)&=(d!)^2V(\Delta_{Y_\bullet}(L_1),\dots,\Delta_{Y_\bullet}(L_d))\cdot \mathrm{vol}(\Delta_{Y_\bullet}(L))\\
	&\le d(d!)^2 V(\Delta_{Y_\bullet}(L),\dots,\Delta_{Y_\bullet}(L),\Delta_{Y_\bullet}(L_d))\\
	&\quad \times V(\Delta_{Y_\bullet}(L_1),\dots,\Delta_{Y_\bullet}(L_{d-1}),\Delta_{Y_\bullet}(L))\\
	&\le d (L^{d-1}\cdot L_d)\cdot (L_1\cdot\ldots\cdot L_{d-1}\cdot L).
\end{align*}
For the last inequality, we used Theorem \ref{thm} together with the inequality
$$d!\cdot V(\Delta_{Y_\bullet}(M_1),\dots,\Delta_{Y_\bullet}(M_1),\Delta_{Y_\bullet}(M_2))\le (M_1^{d-1}\cdot M_2),$$
which we proved for any admissible flag $Y_\bullet$ and any ample line bundles $M_1, M_2$ on $X$ in \cite[Lemma 6.1]{Wil25}. 
\section*{Acknowledgments}
I would like to thank François Ballaÿ and Mingchen Xia for useful discussions.

\end{document}